\documentclass[paper,10pt]{article}

\usepackage{verbatim}
\usepackage{fullpage}
\usepackage{latexsym}
\usepackage{amsfonts, amsthm}
\usepackage[all]{xy}
\usepackage{rotating}
\usepackage{amsmath}
\usepackage{amssymb}
\usepackage{amscd}
\usepackage{enumerate}
\usepackage{mathrsfs}
\usepackage[english]{babel}
\usepackage{makeidx}
\usepackage{hyperref}
\usepackage[utf8]{inputenc}
\usepackage{fancyhdr}
\newtheorem{defi}{Definition}[section]
\newtheorem{prop}[defi]{Proposition} 
\newtheorem{lemma}[defi]{Lemma} 
\newtheorem{teo}[defi]{Theorem}
\newtheorem{cor}[defi]{Corollary} 
\newtheorem{rmk}[defi]{Remark}

\newtheorem*{teo*}{Theorem}

\newcommand{\Z}{\mathbb{Z}}

\newcommand{\Spf}{\mathrm{Spf}}

\newcommand{\Sp}{\mathrm{Spec}}

\providecommand{\bysame}{\leavevmode\hbox to3em{\hrulefill}\thinspace}
\providecommand{\MR}{\relax\ifhmode\unskip\space\fi MR }
\providecommand{\href}[2]{#2}

\begin{document}
\title{\textsc{On $p$-adic differential equations on semistable varieties II}}
\author{Valentina Di Proietto\footnote{Institut de Recherche Math\'ematique 
Avanc\'ee, Universit\'e de Strasbourg, 7 rue Ren\'e Descartes, 
Strasbourg 67084, France.} \, and \ Atsushi Shiho\footnote{Graduate School of 
Mathematical Sciences, the University of Tokyo, 3-8-1 Komaba, Meguro-ku, Tokyo 153-8914, Japan. Mathematics Subject Classification (2010): 12H25, 14F35.}}
\date{}
\maketitle

\textsc{Abstract -} \begin{small} 
This paper is a complement to the paper \cite{DiP12}. 
Given an open variety over a DVR with semistable reduction, 
the author constructed in \cite{DiP12} 
a fully faithful algebraization functor from 
the category of certain log overconvergent isocrystals 
on the special fiber to the category of 
modules with regular integrable connection on the generic fiber. 
In this paper, we prove that, with convenable hypothesis, 
this functor is a tensor functor whose essential image is closed under 
extensions and subquotients. 
As a consequence, we can find suitable Tannakian subcategories of 
log overconvergent isocrystals and of 
modules with regular integrable connection 
on which the algebraization functor is 
an equivalence of Tannakian categories. 
\end{small}
\section{Introduction}

Let $V$ be a complete discrete valuation ring of mixed characteristic $(0,p)$, let $K$ be its fraction field and let $k$ be the residue field. 
Let $X$ be a proper semistable variety over $V$ 
with special fiber $X_k$ and generic fiber $X_K$, endowed with a normal crossing divisor $D$. We denote by $U$ the complement of $D$ in $X$, with special fiber $U_k$ and generic fiber $U_K$. We consider $X$ as a log scheme, with the log structure associated to $D \cup X_k$. 
We consider $\mathrm{Spec}(V)$ also as a log scheme, with the log structure 
associated to the closed point. 

Let $t$ be the number of irreducible components of $D$. 
We call a subset $\Sigma$ of the form 
$\prod_{i=1}^t \Sigma_i$ in $\mathbb{Z}_p^t$ 
(NID) (resp. (NLD)) if, for any $i$ and any $\alpha, \beta \in \Sigma_i$, 
$\alpha -\beta$ is  not a non zero integer (resp. is $p$-adically non-Liouville). 
For $\Sigma \subset \mathbb{Z}_p^t$ which is (NID) and (NLD), 
the author, in \cite{Shilogext}, introduced the notion of $\Sigma$-unipotent 
monodromy, as a $p$-adic analogue for isocrystals of the classical notion 
of regular singularity for modules with 
integrable connection, in the case where $X$ is smooth. 
(In \cite{Shilogext}, we only need the special fibers $X_k, U_k$. 
However, this is not our concern here.) 
In \cite{DiP12} the author generalized this notion 
to the present situation %%case of semistable variety 
and constructed a fully faithful algebraization functor from the category of 
log overconvergent isocrystals on $U_k$ with $\Sigma$-unipotent monodromy 
to the category of modules with integrable connection on $U_K$, 
regular along the generic fiber $D_K$ of $D$. 

If $\Sigma$ is a subgroup of $\Z_p^t$, the above algebraization functor 
would be a tensor functor. However, 
in this case, $\Sigma$ does not contain any non-zero rational number 
because of the (NID) hypothesis, and this would not be a natural condition 
because any log overconvergent isocrystals having geometric origin 
would not belong to the category. 

So we need to consider a more natural condition. 
As observed in \cite{Shilogext}, remark 1.4, 
the notion of $\Sigma$-unipotent monodromy depends only on 
the reduction $\overline{\Sigma}$ 
of $\Sigma$ modulo $\mathbb{Z}^t$. So we can call the notion 
of $\Sigma$-unipotent monodromy as 
$\overline{\Sigma}$-unipotent monodromy. Then 
we can start from a subset $\overline{\Sigma} = \prod_{i=1}^t 
\overline{\Sigma}_i$ of $(\mathbb{Z}_p/\mathbb{Z})^t$ 
satisfying the (NLD) condition (defined in a way similar to above) 
and we can consider the category of log overconvergent isocrystals with 
$\overline{\Sigma}$-unipotent monodromy. If moreover 
$\overline{\Sigma}$ is a subgroup of $(\mathbb{Z}_p/\mathbb{Z})^t$, 
this category is a rigid abelian tensor category. 
Typical example of $\overline{\Sigma}$ is the group 
$(\mathbb{Z}_{(p)}/\mathbb{Z})^t$, in which case 
any log overconvergent isocrystals having geometric origin 
would belong to the category. 

To construct the algebraization functor in this setting, 
we choose a section $\tau:\mathbb{Z}_p/\mathbb{Z}\rightarrow \mathbb{Z}_p$ 
of the canonical projection and we proceed as in \cite{DiP12}: 
It is obtained as the composite of the following four functors. 
\begin{enumerate}
\item[(1)] 
The equivalence of categories between 
the category of log overconvergent isocrystals on $U_k$ 
with 
$\overline{\Sigma}$-unipotent monodromy 
($= \tau(\overline{\Sigma}) := \prod_{i=1}^t \tau(\Sigma_i)$-unipotent 
monodromy) and the category of log convergent isocrystals on $X_k$ 
with exponents in $\tau(\overline{\Sigma})$, which is 
the $p$-adic semistable version of the canonical logarithmic extension of 
Deligne, Andr\'e-Baldassarri. 
\item[(2)] 
The fully faithful functor of `forgetting the convergence condition' 
from the category of log convergent isocrystals on $X_k$ 
with exponents in $\tau(\overline{\Sigma})$ to 
the category of modules with integrable log connection 
on the $p$-adic completion $\hat{X}$ of $X$ with 
exponents in $\tau(\overline{\Sigma})$. 
\item[(3)]
The GAGA equivalence 
between the category of modules with integrable log connection 
on $\hat{X}$ with 
exponents in $\tau(\overline{\Sigma})$ and 
that on %% of modules with integrable log connection
$X_K$ with 
exponents in $\tau(\overline{\Sigma})$. 
\item[(4)] 
The fully faithful restriction functor from 
the category of 
modules with integrable log connection on $X_K$ 
with exponents in 
$\tau(\overline{\Sigma})$ to that of 
modules with regular integrable connection on $U_K$, 
whose essential image lands in the subcategory of objects with 
exponents in $\overline{\Sigma}$. 
%%By logarithmic extension theorem of Deligne and Andr\'e-Baldassarri, 
%%it lands in the category of coherent modules 
%%with integrable connection on $U_K$, regular along $D_K$ with 
%%$\overline{\Sigma}$-unipotent monodromy. 
\end{enumerate}

Note that, since $\tau(\overline{\Sigma})$ is not necessary a group, 
the functor (1) cannot be a tensor functor 
and the composite, which we denote by $\psi_{\tau}$, depends 
a priori on the choice of $\tau$. Nevertheless, we prove in this paper that 
the functor $\psi_{\tau}$ is a tensor functor which is independent of 
$\tau$. We prove this by looking carefully the definition of 
$\psi_{\tau}$ and using techniques developed in 
\cite{Shilogext}, \cite{Shiparalog}. 
We also prove that the essential image of the functor $\psi_{\tau}$ 
is closed by extensions and subquotients. 

As a consequence, we have the surjection of Tannaka duals when 
$U_K$ is connected and admits a $K$-rational point. 
Also, we can find 
Tannakian subcategories of log overconvergent isocrystals on 
$U_k$ with $\overline{\Sigma}$-unipotent monodromy and of modules with regular integrable connection on $U_K$ 
strictly containing the unipotent parts, on which 
the functor $\psi_{\tau}$ becomes an equivalence of categories. 
(It suffices to consider the smallest subcategory containing 
a log overconvergent isocrystal ${\mathcal{E}}$ with 
$\overline{\Sigma}$-unipotent monodromy and closed under 
extensions, subquotients, tensors and duals.) 

The content of each section is as follows. 
In the second section, we give preliminaries and fix some notations 
on our geometric setting. In the third section, 
we recall the definitions and the results in \cite{DiP12} 
concerning log-$\nabla$-modules on 
rigid analytic spaces and isocrystals. 
In the fourth section, we recall the definitions and 
the results in \cite{DiP12}, \cite{AndBal01} 
concerning modules with integrable connection on (formal log) schemes. 
In the fifth section, we give a proof of our results explained above. 

\section{Preliminaries and notation}\label{Preliminaries and notation}
Let $V$ be a complete discrete valuation ring of mixed characteristic $(0,p)$ with uniformizer $\pi$, let $K$ be its fraction field and let $k$ be the residue field.
Let $X$ be a proper semistable variety over $\Sp (V)$ with special fiber $X_k$ and generic fiber $X_K$, and let $D$ be a horizontal normal crossing divisor. This means that \'etale locally we have a cartesian diagram
\begin{equation}\label{etale locally}
\xymatrix{ 
\ D=\bigcup_{j=1}^{s}D_{j}\ \ar@{^(->}[r] \ar[d] &\ X \ar[d]  \\  
\ \bigcup_{j=1}^s\{y_j=0\}\  \ar@{^(->}[r]& \  \Sp (V[x_1,\dots,x_n,y_1,\dots,y_m]/(x_1\dotsm x_r-\pi) )  \\ 
}  
\end{equation}
where the vertical maps are \'etale and the horizontal maps are closed immersions.

We consider on $X$ the log structure $M$ associated to the divisor with normal crossings $X_k\cup D$ in $X$; the structural morphism from $X$ to $\Sp (V)$ extends to a log smooth morphism of log schemes $(X,M)\rightarrow(\Sp(V),N)$, where $\textrm{Spec} (V)$ is endowed with the log structure $N$ associated to the closed point. 
We denote the pull-back of the log structure $M$ (resp. $N$) to 
$X_k$ or $X_K$ (resp. $\Sp (k)$ or $\Sp (K)$) 
also by $M$ (resp. $N$), by abuse of notation. 
Note that the log structure $N$ on $\Sp (K)$ 
is the trivial log structure. Let us put $U := X \setminus D$ and denote 
the canonical open immersion $U \hookrightarrow X$ by $j$. 
We denote the pull-back of 
the log structure $M$ to $U$ again by $M$.

We denote by $\hat{X}$ and $\hat{D}$ the $p$-adic completion of $X$ and $D$, 
respectively. We denote the pull-back of 
the log structure $M$ to $\hat{X}$ again by $M$ and 
the pull-back of 
the log structure $N$ to $\Spf (V)$ again by $N$. 
\'Etale locally we have a diagram analogous to (\ref{etale locally}) for formal schemes:
\begin{equation}\label{etale locally for}
\xymatrix{ 
\ \hat{D}=\bigcup_{j=1}^{s}\hat{D}_{j}\ \ar@{^(->}[r] \ar[d] &\ \hat{X} \ar[d]  \\  
\ \bigcup_{j=1}^s\{y_j=0\}\  \ar@{^(->}[r]& \  \Spf (V\{x_1,\dots,x_n,y_1,\dots,y_m\}/(x_1\dotsm x_r-\pi) ).  \\ 
}  
\end{equation}
Let $\hat{X}_{sing}$ and $\hat{D}_{sing}$ be 
the singular loci of $\hat{X}$ and $\hat{D}$ respectively, 
and we define $\hat{X}^{\circ}, \hat{D}^{\circ}$ as follows: 
$$\hat{X}^{\circ}=\hat{X} \setminus (\hat{X}_{sing}\cup \hat{D}_{sing}), 
\quad 
\hat{D}^{\circ}=\hat{D} \setminus 
(\hat{X}_{sing}\cup \hat{D}_{sing})=\hat{X}^{\circ}\cap \hat{D}.$$
When we consider the situation \'etale locally and fix a diagram (\ref{etale locally for}), we can define the open formal subscheme $\hat{X}_i^{\circ}$ of 
$\hat{X}$ by 
$\hat{X}_i^{\circ} := \{x_{i'} \not= 0 \text{ for } \forall i' \not= i\}$ 
and the open formal subscheme $\hat{X}^{\circ}_{i,j}$ of $\hat{X}_i^{\circ}$ 
by $\hat{X}^{\circ}_{i,j} := \{
x_{i'} \not= 0 \text{ for } \forall i' \not= i, 
y_{j'} \not= 0 \text{ for } \forall j' \not= j\}$. 
Moreover we define $\hat{D}^{\circ}_{i,j}$ by 
$\hat{D}^{\circ}_{i,j} := \hat{X}^{\circ}_{i,j}\cap\hat{D}=
\hat{X}^{\circ}_{i,j}\cap\hat{D}_j$. 
With this notations we have the following relations:
$$\coprod_{i}\hat{X}_i^{\circ}=\hat{X}-\hat{X}_{sing}, \quad 
\coprod_{i,j}\hat{X}^{\circ}_{i,j}=\hat{X}^{\circ}, \quad 
\coprod_{i,j}\hat{D}^{\circ}_{i,j}=\hat{D}^{\circ}.$$
If we denote by the subscript $_K$ the rigid analytic space associated to a formal scheme, then the sets  $\hat{D}^{\circ}_{i,j;K}$ and $\hat{X}^{\circ}_{i,j;K}$ can be described as follows:
$$\hat{X}^{\circ}_{i,j;K}=\{P\in \hat{X}_{K}| \,\forall i'\neq i \,\,\,|x_{i'}(P)|=1 \,\,\,,\forall j'\neq j \,\,\,|y_{j'}(P)|=1\,\},$$
$$\hat{D}^{\circ}_{i,j;K}=\{P\in \hat{X}_{K}| \,\forall i'\neq i \,\,\,|x_{i'}(P)|=1 \,\,\,,\forall j'\neq j \,\,\,|y_{j'}(P)|=1\,\,\,,y_j(P)=0\}.$$
Finally we  denote by $\hat{U}$ the open formal subscheme complement of $\hat{D}$ in $\hat{X}$.

\section{Log-$\nabla$-modules and isocrystals}\label{Log-nabla-modules and isocrystals}
We recall the notion of (log-)$\nabla$-modules on rigid spaces defined and used by Kedlaya in \cite{Ked07} and by the author in \cite{Shilogext}, \cite{Shiparalog}, \cite{Shi11a} and \cite{ShiCUT}. 

\begin{defi}
Let $Y$ be a smooth rigid analytic space over $K$. A $\nabla$-module on $Y/K$ is a coherent $\mathcal{O}_Y$-module $E$ equipped with an integrable connection 
$$\nabla: E\rightarrow E\otimes \Omega^1_{Y/K}$$
with $\Omega^1_{Y/K}$ the sheaf of continuous one forms on the rigid analytic space $Y$. 
\end{defi}

It is known that, for a $\nabla$-module $(E,\nabla)$, $E$ is automatically 
a locally free $\mathcal{O}_Y$-module of finite rank. 

\begin{defi}
Let $Y$ be a smooth rigid analytic space over $K$, with $y_1, \dots, y_n$ $\in$ $\Gamma(Y, \mathcal{O}_Y)$. A log-$\nabla$-module on $Y/K$ with respect to $y_1, \dots, y_n$ is a locally free $\mathcal{O}_Y$-module $E$ of finite rank 
equipped with an integrable connection
$$\nabla: E\rightarrow E\otimes \omega^1_{Y/K}, $$
where 
$$\omega^1_{Y/K}=(\Omega^1_{Y/K} \oplus \bigoplus_{j=1}^n \mathcal{O}_{Y}\mathrm{dlog}y_j)/L$$
with $L$ the coherent sub $\mathcal{O}_{Y}$-module generated by $(\mathrm{d}y_j,0)-(0,y_j \mathrm{dlog}y_j)$ for $1\leq j \leq n.$ \par 
\end{defi}

Contrary to the case of $\nabla$-modules, the local freeness should be put in 
the definition in the case of log-$\nabla$-modules. 

\begin{rmk}\label{log-nonlog}
When $y_1, \dots, y_n$ are invertible, the notion of a log-$\nabla$-module 
on $Y/K$ with respect to $y_1, \dots, y_n$ is the same as 
the notion of a $\nabla$-module on $Y/K$. 
\end{rmk}

For $j$ $\in$ $\{1,\dots, n\}$ there is a natural map 
$$\Omega^1_{Y/K} \oplus \bigoplus_{j' \neq j}\mathcal{O}_{Y}\textrm{dlog}y_j\rightarrow \omega^{1}_{Y/K}, $$
whose image we denote by $M_{j}$. 
If the zero loci $D_{j}$ 
of $y_j \,(1 \leq j \leq n)$ in $Y$ are smooth and meet transversally, 
we have the isomorphisms $\omega^{1}_{Y/K}/M_j \cong 
\mathcal{O}_{D_j} \textrm{dlog}y_j \cong \mathcal{O}_{D_j}$. 

\begin{defi}\label{analyticexponents}
Let $Y, y_1, \dots, y_n$ be as above and assume that 
the zero loci $D_{j}$ 
of $y_j \,(1 \leq j \leq n)$ in $Y$ are affinoid, smooth 
and meet transversally. For a log-$\nabla$-module $(E,\nabla)$ on $Y$, 
let $\textrm{res}_j$ be the endomorphism on 
$E \otimes \mathcal{O}_{D_j}$ induced by 
$$E\rightarrow E \otimes \omega^1_{Y/K} \rightarrow 
E\otimes \omega^1_{Y/K}/M_j \cong 
E\otimes \mathcal{O}_{D_j} \mathrm{dlog}y_j \cong E \otimes \mathcal{O}_{D_j}. 
$$
We call it the residue of $E$ along $D_{j}.$ 
Thanks to proposition-definition 1.24 in \cite{Shi11a} (based on 
proposition 1.5.3 of \cite{BalChi93}), we know that 
there exists a minimal monic polynomial $P_j$ $\in$ $K[T]$ 
such that $P_j(\mathrm{res}_j)=0$. We call the roots of $P_j$ 
the exponents of $E$ along $D_{j}$.
\end{defi}

\begin{rmk}\label{analyticexponents-patching}
It is easy to see that the above definition works also in the case 
where $Y$ has a finite admissible covering $Y = \bigcup_{i=1}^m Y_i$ such 
that each $Y_i \cap D_j$ are affinoids: Indeed, we can define 
the minimal monic polynomial in the definition 
for each $Y_i \cap D_j$, and it suffices to define the polynomial $P_j$ as 
the least common multiple of them. 
\end{rmk}

\begin{comment}
\begin{defi}
Let the notations be as above. 
For $\Sigma=\prod_{j=1}^n\Sigma_j\subset \mathbb{Z}_p^n$, we denote by $\mathrm{LNM}_{Y, \Sigma}$ the category of log-$\nabla$-modules on $Y$ with respect to $y_1, \dots, y_n$ whose exponents along $D_{j}$ are in $\Sigma_j$ for all $j$. 
\end{defi}
\end{comment}

For an aligned interval (an interval $I$ such that any endpoint is contained in $\Gamma^{*} := \sqrt{|K^*|} \cup \{0\}$), we define an $n$-dimensional rigid analytic polyannulus $A^n_K(I)$ over $K$ with radius in $I$ 
as 
$$A^n_K(I)=\{(t_1,\dots,t_n) \in \mathbb{A}^{n,rig}_K| |t_j|\in I \,\forall \,\, j=1,\dots,n\}.$$ 
If $\Sigma=\prod_{j=1}^n\Sigma_j$ is a subset of $\mathbb{Z}_p^n$ and $\xi :=(\xi_1, \dots,\xi_n)$ an element of $\Sigma$, we denote by  $M_{\xi}$ the log-$\nabla$-module on $A^n_K(I)$ defined by $(\mathcal{O}_{A^n_K(I)}, d+\sum_{j=1}^n\xi_i \textrm{dlog}t_{j})$.

\begin{defi}\label{sigmaunipotence}
Let $Y$ be a smooth rigid analytic space, $I$ an aligned interval and $\Sigma=\prod_{j=1}^{n}\Sigma_j$ $\subset$ $\mathbb{Z}_p^{n}$. We regard that the smooth rigid analytic space $Y \times A_K^n(I)$ is endowed with 
the sections $t_1, \dots, t_n$, which are the coordinates of the polyannulus 
$A_K^n(I)$. We say that a log-$\nabla$-module $E$ on 
$Y\times A^n_K(I)$ is $\Sigma$-unipotent if there exists a filtration 
$$0\subset E_1 \subset \dots \subset E_t=E$$
by subobjects such that every successive quotient $E_i/E_{i-1}$ has the form 
$\pi_1^*{F}\otimes \pi_2^*(M_{\xi})$, where $\pi_1:Y\times A^n_K(I)\rightarrow Y$ denotes the first projection, $\pi_2:Y\times A^n_K(I)\rightarrow A^n_K(I)$ the second, $F$ is a $\nabla$-module on $Y$ and $M_{\xi}$ the log-$\nabla$-module we defined before with $\xi$ $\in$ $\Sigma$. 
%%%We denote by $\mathrm{ULNM}_{Y\times A^n_K(I), \Sigma}$ the category of $\Sigma$-unipotent log-$\nabla$-modules on $Y\times A^n_K(I).$
\end{defi}

\begin{rmk}
As observed in \cite{Shilogext} 1.4, if $I$ does not contain $0,$ the notion of $\Sigma$-unipotence depends only on the image $\overline{\Sigma}$ of $\Sigma$ in $(\mathbb{Z}_p/\mathbb{Z})^n$ in the following sense: any 
log-$\nabla$-module $(E, \nabla)$ in $Y\times A^n_K(I)$ is $\Sigma$-unipotent if and only if for any section $\tau=\prod_{j=1}^n\tau_j:(\mathbb{Z}_p/\mathbb{Z})^n\rightarrow (\mathbb{Z}_p)^n$ of the canonical projection is $\tau(\overline{\Sigma})$-unipotent. So, in the following, we often call the $\Sigma$-unipotence as 
the $\overline{\Sigma}$-unipotence. 
\end{rmk}
As in definition 1.8 of \cite{Shilogext} we can give following definition.
\begin{defi}
A set $\Sigma$ $\subset$ $\mathbb{Z}_p$ is called (NID) (resp. (NLD)) if for any $\alpha$, $\beta$ $\in$ $\Sigma$, $\alpha-\beta$ is not a non zero integer (resp. is p-adically non-Liouville). A set $\Sigma=\prod_{j=1}^{n}\Sigma_{j}$ $\subset$ $\mathbb{Z}_p^{n}$ is called (NID) (resp. (NLD)) if 
$\Sigma_{j}$ is (NID) (resp. (NLD)) for any $j=1,\dots,n$. \par 
A set $\Sigma$ $\subset$ $\mathbb{Z}_p/\mathbb{Z}$ is called (NLD) if, 
for any section $\tau: \mathbb{Z}_p/\mathbb{Z} \rightarrow \mathbb{Z}_p$ 
of the projection and for any $\alpha$, $\beta$ $\in$ $\Sigma$, 
$\alpha-\beta$ is $p$-adically non-Liouville. 
A set $\Sigma=\prod_{j=1}^{n}\Sigma_{j}$ $\subset$ 
$(\mathbb{Z}_p/\mathbb{Z})^n$ is (NLD) if 
$\Sigma_{j}$ is (NLD) for any $j=1,\dots,n$. 
\end{defi}

\begin{comment}
We  use the following result (\cite{Ked07} 3.3.4, \cite{Ked07} 3.3.6, \cite{Shilogext} corollary 1.15 and \cite{Shilogext} corollary 1.16). 
\begin{teo}\label{unipotentiequivalenti}
Let $Y$ be as before, $I$ a quasi open interval (open at any non-zero endpoint) and $\Sigma=\prod_{j=1}^{n}\Sigma_{j}$ $\subset$ $\bar{K}^{n}$ that is (NID) and (NLD), then  we have an equivalence of categories 
$$\mathcal{U}_I:\mathrm{ULNM}_{Y\times A^n_K([0,0]),\Sigma}\rightarrow \mathrm{ULNM}_{Y\times A^n_K(I),\Sigma}, $$
where $Y\times A^n_K([0,0])$ denotes the smooth rigid analytic space 
$Y$ endowed with the sections $0, ..., 0$ ($n$ times), and 
$\mathcal{U}_I$ is the functor which associates to a $\Sigma$-unipotent log-$\nabla$-module $E$ on $Y\times A^n_K([0,0])$ a $\Sigma$-unipotent log-$\nabla$-module $\mathcal{U}_I(E)$ on $Y\times A^n_K(I)$ defined as the sheaf $\pi^*E$ (where $\pi_1: Y \times A^n_K(I)\rightarrow Y$ denotes the projection) and the connection 
$$v\mapsto \pi_1^*(\nabla)v+\sum_{i=1}^n\pi_1^*(t_i\frac{\partial}{\partial t_i})(v)\otimes \frac{dt_i}{t_i}.$$
If $I$ is an interval of positive length which is not necessarily quasi open, 
the functor $\mathcal{U}_I$ is fully faithful.
\end{teo}  
\end{comment}

Now we recall the theorem of logarithmic extension for log overconvergent 
isocrystals in semistable situation proven in theorem 5, section 12 of \cite{DiP12}, which generalizes of the main theorem of \cite{Shilogext}. We start recalling the notion of log overconvergent isocrystal with $\Sigma$-unipotent monodromy, and the notion of log convergent isocrystal with exponents in $\Sigma$.

We consider $I^{\dag}(U_k/V)^{\log}$, the category of overconvergent log isocrystals on the log pair $((U_k,M),(X_k,M))$ over $(\Spf (V),N)$, as defined in section 4 in \cite{Shi07b}. (See section 10 in \cite{DiP12} for the definition in the present situation. Note that we use simpler notation in this paper to lighten the notation.) In the present situation, it is defined 
as a coherent $j^{\dagger}\mathcal{O}_{\hat{X}_K}$-module endowed with 
a stratification on the log tubular neighborhood 
$]X[_{(\hat{X},M) \times_{(\Spf (V),N)} (\hat{X},M)}^{\log}$ satisfying 
a certain cocycle condition. 
Given an object $\mathcal{E}$ of $I^{\dag}(U_k/V)^{\log}$ and a strict neighborhood $W$ of $]U_k[_{\hat{X}}$ in $]X_k[_{\hat{X}} = \hat{X}_K$ 
which is small enough, $\mathcal{E}$ induces 
a $\nabla$-module $(E,\nabla)$ on $W$.
(This follows from the fact that the log structure $M$ on $\hat{U}$ comes only from the special fiber $U_k$ of $U$, not from $\hat{D}$.)  So $E$ is locally free, and this implies that 
$\mathcal{E}$ is automatically a locally free 
$j^{\dagger}\mathcal{O}_{\hat{X}_K}$-module of finite rank. 
%%a module $E$ with an integrable connection 
%%$$\nabla: E\rightarrow E\otimes \omega^1_{W/K}, $$
%%where $\omega^1_{W/K}$ is the restriction of $K\otimes \omega^1_{(\hat{X},M)/(%%\Spf(V),N)}$ to $W$.

If \'etale locally we are in the situation described in (\ref{etale locally for}), $W$ contains $\hat{D}_{j,i;K}^{\circ}\times A^1_K([\lambda,1))$ for all $i, j $ and for some $\lambda$ $\in $ $(0,1)\cap \Gamma^{*}$ and so we can 
define the restriction of $(E,\nabla)$ on it, which is 
a log-$\nabla$-module on 
$\hat{D}_{j,i;K}^{\circ}\times A^1_K([\lambda,1))$ (for the reason that 
we can put the term `log-' here, see remark \ref{log-nonlog}). 

Let $\overline{\Sigma}=\prod_{h=1}^t \overline{\Sigma}_h$ be a subset of $(\mathbb{Z}_p/\mathbb{Z})^t$ which is (NLD), where $t$ is the number of irreducible components of 
$D=\bigcup_{h=1}^t D^h$. If there exists a small enough \'etale covering $\coprod_l\phi_l:\coprod_l \hat{X}_l \rightarrow \hat{X}$ such that every $\hat{X}_l$ has a diagram 
\begin{equation}\label{etale locally for l}
\xymatrix{ 
\ \hat{D}_l=\bigcup_{j=1}^{s}\hat{D}_{j,l}\ \ar@{^(->}[r] \ar[d] &\ \hat{X}_l \ar[d]  \\  
\ \bigcup_{j=1}^s\{y_{j,l}=0\}\  \ar@{^(->}[r]& \  \Spf V\{x_{1,l},\dots,x_{n,l},y_{1,l},\dots,y_{m,l}\}/(x_{1,l}\dots x_{r,l}-\pi)   \\ 
}  
\end{equation}
as in (\ref{etale locally for}), then we can define a function of sets $h: \{1,\dots,s\}\rightarrow \{1,\dots,t\}$ for each $l$ as follows: 
for any $1\leq j\leq s$, $\phi_l(\hat{D}_{j,l})$ is contained in 
one irreducible component $D^{h(j)}$ of $D$. 
We denote by $\overline{\Sigma}_{h(j)}$ the factor of $\overline{\Sigma}$ corresponding to the component $D^{h(j)}$.
\begin{defi}\label{unipotentmonodromy}
A log overconvergent isocrystal $\mathcal{E}$ has $\overline{\Sigma}$-unipotent monodromy if there exists an \'etale covering $\coprod_l\phi_l:\coprod_l \hat{X}_l \rightarrow \hat{X}$ such that every $\hat{X}_l$ has a diagram
(\ref{etale locally for l}) 
with $\hat{D}_l =\phi_l^{-1}(\hat{D})$ and the restriction of the $\nabla$-module %% $E_{\mathcal{E},l}$ 
associated to $\mathcal{E}$ to 
$\hat{D}^{\circ}_{i,j,l;K}\times A^1_K([\lambda,1))$ 
is $\overline{\Sigma}_{h(j)}$-unipotent for all $l$, $i$ and $j$. 
%% for which $\phi_l(\hat{D}_{i,j,l}^{\circ})\subset \hat{D}^{h(i,j)}$. 
We denote by $I^{\dag}(U_k/V)^{\log, \overline{\Sigma}}$ 
the category of log overconvergent isocrystals with 
$\overline{\Sigma}$-unipotent monodromy .
\end{defi}
\begin{rmk}
Here we used a slightly different formulation from 
\cite{DiP12}: $t$ there was the number of irreducible components of 
$\hat{D}$, and the condition `$\overline{\Sigma}_{h(j)}$-unipotent' here was 
`$\cap_{i=1}^r {\Sigma}_{h(i,j)}$-unipotent' in the notation there. 
However, this does not make any serious difference. 
\end{rmk}

The notion of $\overline{\Sigma}$-unipotent monodromy is independent of the choice of the \'etale covering and the diagram 
(\ref{etale locally for l}) which are chosen in definition \ref{unipotentmonodromy} (proposition 17 of section 12 of \cite{DiP12}).

We also consider the category of locally free log convergent isocrystals as defined in \cite{Shi00} section 5. (See sections 4 and 8 in \cite{DiP12} for the discussion in the present situation.) 
It is the category of locally free isocrystals on the log convergent site $((X_k, M)/(\Spf (V),N))_{\mathrm{conv}}$; we denote it by 
$I_{\mathrm{conv}}(X_k/V)^{\log}.$  In the local situation as in (\ref{etale locally for}), $\mathcal{E}\in I_{\mathrm{conv}}(X_k/V)^{\log}$ induces a log-$\nabla$-module on $\hat{X}_K$ with respect to $y_1,\dots y_s.$ 

Let $\overline{\Sigma}=\prod_{h=1}^t \overline{\Sigma}_h$ be a subset of $(\mathbb{Z}_p/\mathbb{Z})^t$, where $t$ is the number of the irreducible components of $D=\bigcup_{h=1}^tD^h$ and let $\tau:\mathbb{Z}_p/\mathbb{Z}\rightarrow \mathbb{Z}_p$ be a section of the canonical surjection $\mathbb{Z}_p \rightarrow \mathbb{Z}_p/\mathbb{Z}$. 

\begin{defi}\label{residueanalytic} 
A locally free convergent isocrystal $\mathcal{E}$ has exponents 
%% along $\hat{D}_K$ 
in $\tau({\overline{\Sigma}}) := \prod_{h=1}^t \tau(\overline{\Sigma}_h)$ 
if there exists an \'etale covering $\coprod_l\phi_l:\coprod_l \hat{X}_l \rightarrow \hat{X}$ such that every $\hat{X}_l$ has a diagram 
\eqref{etale locally for l} with $\hat{D}_l:=\phi_l^{-1}(\hat{D})$ and the log-$\nabla$-module $E_{\mathcal{E},l}$ on $\hat{X}_{l;K}$ induced by $\mathcal{E}$ has exponents 
in $\tau(\overline{\Sigma}_{h(j)})$ along $\hat{D}_{j,l;K}$ for all $l$ and 
$j$. We denote by $I_{\mathrm{conv}}(X_k/V)^{\log,\tau(\overline{\Sigma})}$ the category of locally free log convergent isocrystals with exponents in $\tau(\overline{\Sigma})$.
\end{defi}
The notion of exponents in $\tau({\overline{\Sigma}})$ is independent of the choice of the \'etale covering and the diagram (\ref{etale locally for l}), which are chosen in definition \ref{residueanalytic} (lemma 2 of section 8 of \cite{DiP12}). \\

\begin{defi}\label{NLG}
Let $\overline{\Sigma}=\prod_{h=1}^t \overline{\Sigma}_h$ be a subset of $(\mathbb{Z}_p/\mathbb{Z})^t$. We say that $\overline{\Sigma}$ is (NLG) if it is a subgroup of $(\mathbb{Z}_p/\mathbb{Z})^t$ and 
if for any section $\tau:\mathbb{Z}_p/\mathbb{Z}\rightarrow \mathbb{Z}_p$ the set $\tau(\overline{\Sigma}):=\prod_{h=1}^t\tau(\overline{\Sigma}_h)$ consists of elements that are, component by component, $p$-adic non-Liouville numbers.
\end{defi}

Let us remark that if $\overline{\Sigma}$ is (NLG) and if 
$\tau:\mathbb{Z}_p/\mathbb{Z}\rightarrow \mathbb{Z}_p$ is a section, 
$\tau(\overline{\Sigma})$ is (NID) and (NLD). 
Hence, as in remark 3.17 of \cite{Shilogext}, 
theorem 5 of section 12 of \cite{DiP12} implies the following theorem: 

\begin{teo}\label{mainstable}
We fix a set $\overline{\Sigma}=\prod_{h=1}^t\overline{\Sigma}_h$ $\in$ $(\mathbb{Z}_p/\mathbb{Z})^t$, where $t$ is the number of the irreducible components of $\hat{D}=\bigcup_{h=1}^t\hat{D}^t$ in $\hat{X}$ and we require that $\overline{\Sigma}$ is (NLG). Then, for each section $\tau: \mathbb{Z}_p/\mathbb{Z}\rightarrow \mathbb{Z}_p$, the restriction functor
$$j_{\tau}^{\dag}:I_{\mathrm{conv}}(X_k/V)^{\log, \tau(\overline{\Sigma})}\longrightarrow I^{\dag}(U_k/V)^{\log, \overline{\Sigma}}, $$
which is induced by the tensor functor 
$$j^{\dag}:I_{\mathrm{conv}}(X_k/V)^{\log}
\longrightarrow I^{\dag}(U_k/V)^{\log}, $$ 
is a well-defined functor which is an equivalence of categories.
\end{teo}
\begin{defi}\label{expoperE}
Let the hypothesis be as in theorem \ref{mainstable} and let $\mathcal{E} \in I^{\dag}(U_k/V)^{\log, \overline{\Sigma}}$. 
We define $\overline{\Sigma}_{\mathcal{E}} \subset \overline{\Sigma}$ as 
the minimum subset of the form $\overline{\Sigma}_{\mathcal{E}}=\prod_{h=1}^t\overline{\Sigma}_{\mathcal{E},h}$ such that $\mathcal{E}$ has 
$\overline{\Sigma}_{\mathcal{E}}$-unipotent monodromy. 
%there exists an \'etale covering $\coprod_l\phi_l:\coprod_l \hat{X}_l \rightarrow \hat{X}$ as in definition \ref{unipotentmonodromy} such that the restriction of the log-$\nabla$-module $E_{\mathcal{E},l}$ on $\hat{X}_{l;K}$ to $ \hat{D}^{\circ}_{i,j,l;K}\times A^1_K[\lambda,1)$ is $\cap_{i=1}^r\overline{\Sigma}_{h(i,j)}$-unipotent, for all $i$ and $j$ for which $\phi_l(\hat{D}_{i,j,l}^{\circ})\subset \hat{D}^{h(i,j)}$. We define $\overline{\Sigma}_{\mathcal{E}}=\prod_{h=1}^t\overline{\Sigma}_{\mathcal{E},h}\subset (\mathbb{Z}_p/\mathbb{Z})^t$ the subset of $\overline{\Sigma}$ consisting of the elements $\xi=(\xi_1, \dots, \xi_t)$ such that for every $\xi_h$ there exists $l, i, j$ such that the log-$\nabla$-module $M_{\xi_h}$ appears in the successive quotients of the filtration which makes $E_{\mathcal{E},l}$ on $ \hat{D}^{\circ}_{i,j,l;K}\times A^1_K[\lambda,1)$ a log-$\nabla$-module $\cap_{i=1}^r\overline{\Sigma}_{h(i,j)}$-unipotent. 
\end{defi}
\begin{rmk}
For a given log overconvergent isocrystal $\mathcal{E}$ the set $\overline{\Sigma}_{\mathcal{E}}$ is a finite set. Indeed, we can suppose that the \'etale covering $\coprod_l\phi_l:\coprod_l \hat{X}_l \rightarrow \hat{X}$ of definition \ref{unipotentmonodromy} is finite, hence we have a finite number of annuli 
$ \hat{D}^{\circ}_{i,j,l;K}\times A^1_K([\lambda,1))$, and a finite number of $M_{\xi}$'s appears in the expression in definition 
\ref{sigmaunipotence} for the log-$\nabla$-module on each 
$\hat{D}^{\circ}_{i,j,l;K}\times A^1_K([\lambda,1))$ induced by $\mathcal{E}$. 
\end{rmk}

\begin{defi}
Let $\Sigma_1:=\prod_{j=1}^n\Sigma_{1,j}$ and $\Sigma_2 :=\prod_{j=1}^n\Sigma_{2,j}$ be subsets of $\mathbb{Z}_p^n$; we denote the set
$$\prod_{j=1}^n\{\xi_{1,j}\pm\xi_{2,j}|\, \xi_{1,j} \in \Sigma_{1,j}\,,\, \xi_{2,j} \in \Sigma_{2,j} \}$$
by $\Sigma_1\pm\Sigma_2$. 
For $\boldsymbol{\alpha}=(\alpha_1, \dots, \alpha_n) \in \mathbb{Z}_p^n$, we 
put $$\Sigma \pm \boldsymbol \alpha=:\prod_{j=1}^n\{ \xi_{j}\pm \alpha_j|\, \xi_{j} \in \Sigma_{j} \}.$$
\end{defi}

We see that $I^{\dag}(U_k/V)^{\log, \overline{\Sigma}}$ is a rigid abelian tensor category under convenable hypothesis. 
\begin{prop}\label{abelian tensor}
If $\overline{\Sigma}$ is (NLG), the category $I^{\dag}(U_k/V)^{\log, \overline{\Sigma}}$ is a rigid abelian tensor category.
\end{prop}
\begin{proof}
First, the category $I^{\dag}(U_k/V)^{\log}$ is an abelian tensor category having internal hom's; this can be seen easily from definition (of an object as a coherent module endowed with stratification). 
Moreover, using the fact that the objects are locally free, 
one can see that it is rigid. 

Hence it suffices to prove that the category $I^{\dag}(U_k/V)^{\log, \overline{\Sigma}}$ is closed by the formation of subquotient, tensor product and dual in $I^{\dag}(U_k/V)^{\log}$. 

Since the definition of $\overline{\Sigma}$-unipotent monodromy is given \'etale locally, it is sufficient to check these properties for the $\overline{\Sigma}$-unipotence of log-$\nabla$-modules on $\hat{D}^{\circ}_{i,j,l;K}\times A^1_K([\lambda,1))$ with the same notation as in definition \ref{unipotentmonodromy}. 

The notion of $\overline{\Sigma}$-unipotence is closed by the formation of subquotients by proposition 1.17 of \cite{Shilogext}. So we need to prove that the notion of $\overline{\Sigma}$-unipotence is closed by the formation of tensor product and dual. One can check it easily by using the isomorphisms 
$M_{\xi}\otimes M_{\xi'}=M_{\xi+\xi'}$, $M_{\xi}^* = M_{-\xi}$ and the assumption that $\overline{\Sigma}$ is a subgroup of $(\mathbb{Z}_p/\mathbb{Z})^t$.

\end{proof}

\section{Formal and algebraic modules with integrable connection}

In this section, we recall the definition of the category of 
modules with integrable connection on (formal log) schemes. 
First we recall the definition in the case of formal (log) schemes. 

\begin{defi}
We define the category $\widehat{MIC}(\hat{U}/V)$ of 
modules with integrable connection on 
$\hat{U}$ over $V$ as follows: An object 
is an isocoherent sheaf $E$ 
(the sheaf of the form $K \otimes_V \tilde{E}$, where $\tilde{E}$ is a 
coherent $\mathcal{O}_{\hat{U}}$-module) 
endowed with an integrable connection
$$\nabla:E\rightarrow E\otimes \Omega^1_{\hat{U}/V}, $$
where $\Omega^1_{\hat{U}/V}$ is the sheaf of continuous one forms on 
$\hat{U}$ over $V$. 
It is known that $E$ is automatically `locally free' (a projective 
$K \otimes_V \mathcal{O}_{\hat{U}}$-module locally in Zariski topology). 
\end{defi}

\begin{defi}
We define the category $\widehat{MIC}(\hat{X}/V)^{\log}$ of 
modules with integrable connection on 
$(\hat{X},M)$ over $(\Spf (V), N)$ as follows: An object in 
$\widehat{MIC}(\hat{X}/V)^{\log}$ 
is a locally free (in the sense of the above definition) 
isocoherent sheaf $E$ 
%(the sheaf of the form $K \otimes_V \tilde{E}$, where $\tilde{E}$ is a 
%coherent $\mathcal{O}_{\hat{X}}$-module, which is a projective 
%$K \otimes_V \mathcal{O}_{\hat{X}}$-module locally in Zariski topology) 
endowed with an integrable connection
$$\nabla:E\rightarrow E\otimes \omega^1_{(\hat{X}, M)/(V,N)}, $$
where $\omega^1_{(\hat{X}, M)/(V,N)}$ is the sheaf of continuous 
one forms of $(X,M)$ over $(\Spf (V),N)$ 
in the framework of fine log formal schemes. 
\end{defi}

Given $(E, \nabla) \in \widehat{MIC}(\hat{X}/V)^{\log}$ and 
given a diagram (\ref{etale locally for}), we can define the exponents of 
$(E,\nabla)$ along each $\hat{D}_{j}$: the definition is perfectly analogous to the one given for log-$\nabla$-modules on rigid analytic spaces (definition \ref{analyticexponents}). We can globalize the definition of exponents as in definition \ref{residueanalytic} in the following way. 

\begin{defi}
We fix a set $\Sigma=\prod_{h=1}^t \Sigma_h \subset \bar{K}^t$, 
where $t$ is the number of irreducible components of $D = \cup_{h=1}^t D^h$. 
We say that an object $(E,\nabla)$ 
in $\widehat{MIC}(\hat{X}/V)^{\log}$ has exponents in $\Sigma$ 
if there exists an \'etale covering $\coprod_l\phi_{l}: \coprod_l \hat{X}_{l} \rightarrow \hat{X}$ such that every $\hat{X}_{l}$ has a diagram
\eqref{etale locally for l}
with $\hat{D}_{l} = \phi_l^{-1}(\hat{D})$ 
and for every $l$ and $j$, the restriction of 
$(E,\nabla)$ to $\hat{X}_{l}$ 
has exponents in $\Sigma_{h(j)}$ along $\hat{D}_{j,l}$, where 
$\phi_{l}(\hat{D}_{j,l}) \subset D^{h(j)}$. 
We denote by $\widehat{MIC}(\hat{X}/V)^{\log, \Sigma}$ 
the category of 
modules with integrable connection on 
$(\hat{X},M)$ over $(\Spf (V), N)$ with exponents in $\Sigma$. 
\end{defi}

One can prove by arguing as in the analytic case (\cite{DiP12} lemma 2 of section 8) that the above definition does not depend on the choice of the \'etale covering and the diagram (\ref{etale locally for l}). 

\begin{rmk}
If we are 
given a diagram (\ref{etale locally for}), 
the category $\widehat{MIC}(\hat{X}/V)^{\log}$ is naturally 
equivalent to the category of log-$\nabla$-modules on 
$\hat{X}_K$ with respect to $y_1, \dots, y_s$. Under this 
equivalence, the notions of `exponents in $\Sigma$' coincide. 
\end{rmk}

Next we recall the definition in the case of usual (log) schemes. 

\begin{defi}
We define the category $MIC(U_K/K)$ of modules 
with integrable connection on the scheme $U_K$ as follows: 
An object in $MIC(U_K/K)$ is a coherent $\mathcal{O}_{X_K}$-module $E$ 
endowed with an integrable connection
$$\nabla:E\rightarrow E\otimes \Omega^1_{U_K/K}. $$
It is known that $E$ is automatically locally free. 
\end{defi}

\begin{defi}
We define the category $MIC(X_K/K)^{\log}$ of modules with integrable connection on $(X_K,M)$ over $K$ as follows: An object in 
$MIC(X_K/K)^{\log}$ is a locally free $\mathcal{O}_{X_K}$-module $E$ of finite rank endowed with an integrable connection
$$\nabla:E\rightarrow E\otimes \omega^1_{(X_K, M)/K}, $$
where $\omega^1_{(X_K, M)/K}$ is the sheaf of logarithmic one forms 
on $(X_K,M)$ over $K$ in the framework of fine log schemes. 
\end{defi}

As in the formal case, if we are 
given an object $(E, \nabla)$ in $MIC(X_K/K)^{\log}$ and given \'etale locally a diagram as in $K \otimes_V$(\ref{etale locally}), 
we can define the exponents of 
$(E,\nabla)$ along each $D_{j;K} := D_j \otimes_V K$. 
Also, we can globalize the definition of exponents as 
in definition \ref{residueanalytic} in the following way. 

\begin{defi}
We fix a set $\Sigma=\prod_{h=1}^t \Sigma_h \subset \bar{K}^t$, where $t$ is the number of irreducible components of the divisor 
$D_K=\cup_{h=1}^tD^{h;K}$ (which is equal to the number of irreducible 
components of $D$). We say that an object $(E,\nabla)$ 
in $MIC((X_K,M)/K)$ has exponents in $\Sigma$ 
if there exists an \'etale covering $\coprod_l\phi_{l;K}: \coprod_l X_{l;K} \rightarrow X_K$ such that every $X_{l;K}$ has a diagram
\begin{equation}\label{etale locally K}
\xymatrix{ 
\ D_{l;K} = \bigcup_{j=1}^{s} D_{j,l;K}\ \ar@{^(->}[r] \ar[d] &\ X_{l;K} \ar[d]  \\  
\ \bigcup_{j=1}^s\{y_{l,j}=0\}\  \ar@{^(->}[r]& \  \Sp K[x_{l,1},\dots,x_{l,n},y_{l,1},\dots,y_{l,m}]/(x_{l,1}\dots x_{l,r}-\pi)   \\ 
}  
\end{equation}
as in $K \otimes_V$(\ref{etale locally}) with $D_{l;K} := \phi_l^{-1}(D_K)$ and for for every $l$ and $j$, the restriction of 
$(E,\nabla)$ to $X_{l;K}$ 
has exponents in $\Sigma_{h(j)}$ along $D_{j,l;K}$, where 
$\phi_{l;K}(D_{j,l;K})\subset D^{h(j);K}$. We denote by 
$MIC(X_K/K)^{\log, \Sigma}$ the category of modules with integrable connection on $(X_K,M)$ over $K$ with exponents in $\Sigma$.
\end{defi}
Also in this case, one can prove that the above definition does not depend on the choice of the \'etale covering and the diagram (\ref{etale locally K}). \\ 

In \cite{Del70} II, \cite{AndBal01} I.3 and \cite{Bal05} section 2, the notion of regularity along a normal crossing divisor which is union of smooth components is introduced. Since the irreducible components of $D_K$ are not necessarily smooth, we need a slight generalization of this notion. 
\begin{defi} \label{regular local}
We say that an object $(E,\nabla)$ in $MIC(U_K/K)$ 
is regular along $D_K$ if, around every generic point of $D_K$, 
there exists an etale neighborhood 
$X_{l;K} \rightarrow X$ and a diagram as in (\ref{etale locally K})
such that the module with integrable coonection induced on $U_{l; K}:=X_{l; K}\setminus D_{l; K}$ is regular along $D_{l; K}$. 
\end{defi}
\begin{prop}\label{regindip}
The notion of regularity along $D_K$ is independent of the choice of \'etale neighborhood and the diagram chosen in definition \ref{regular local}.
\end{prop}
\begin{proof}
It is essentially proved in \cite{AndBal01} I. 3.4.4, 3.4.5. 
Indeed, let us take another etale neighborhood 
$X'_{l;K} \rightarrow X$ and a diagram as in 
(\ref{etale locally K}) (which we denote by the same symbols with ${}'$). 
We want to prove that the pull-back of $(E,\nabla)$ to $U'_{l; K}:=X'_{l; K}\setminus D'_{l; K}$ is regular along $D'_{l; K}$. If we consider the fiber product $U_{l; K}\times_{U_{K}} U'_{l; K}$, the pull-back of $(E,\nabla)$ to it 
is regular along the pull-back of $D_K$ by 
proposition I. 3.4.4 of \cite{AndBal01}, and it implies the regularity of 
the pull-back of $(E,\nabla)$ to $U'_{l; K}$ again by 
proposition I. 3.4.4 of \cite{AndBal01}. 
%%%Let us suppose that $E$ is regular along $D_K$, hence there exists an \'etale neighborhood $X_{l; K}$ of every generic point of $D_K$ and a diagram analogous to (\ref{etale locally}) such that for every $l$ the module with integrable connection induced on $U_{l; K}:=X_{l; K}\setminus D_{l; K}$ is regular along $D_{l; K}$. We fix a generic point of $D_K$, an  \'etale neighborhood $X_l$ of it and a diagram analogous to (\ref{etale locally}). Using the log extension theorem of André and Baldassarri  (\cite{AndBal01} I. 4.9) we can extend $E$ to $X_{l; K}$ with logarithmic singularities along $D_{l; K}$ with certain exponents along $D_{l; K}$. Since the notion of module with connection with given exponents does not depend on the diagram chosen, then also the notion of regularity does not depend on the diagram chosen. To prove that it is independent of the choice of \'etale neighborhood we pick $X'_l$ with diagram as in (\ref{etale locally}) an other \'etale neighborhood of the geometric point of $D_K$ we chose. 
\end{proof}
There is a notion of exponents for an object of $MIC(U_K/K)$ regular along $D_K$, which is defined in \cite{AndBal01} I. 6.3. As before we need to adapt this notion to our situation. 
\begin{defi}\label{thetabar}
Let $(E, \nabla)$ be an object in $MIC(U_K/K)$ regular along $D_K$. We fix a set $\overline{\Sigma}=\prod_{h=1}^t \overline{\Sigma}_h \subset \bar{K}^t/\mathbb{Z}^t$, where $t$ is the number of irreducible components of the divisor $D_K=\cup_{h=1}^tD^{h;K}$. We say that $E$ has exponents in $\overline{\Sigma}$ if, around every generic point of $D_K$, 
 there exists an \'etale neighborhood $X_{l;K} \rightarrow X_K$ and 
a diagram (\ref{etale locally K}) 
such that the module with integrable connection induced on $U_{l; K}:=X_{l; K}\setminus D_{l; K}$ has exponents in the sense of \cite{AndBal01} I. 6.3 along $D_{j,l;K}$ contained in $\overline{\Sigma}_{h(j)}$ for every $j$, where $\phi_{l;K}(D_{j,l;K})\subset D^{h(j);K}$. We denote the category of modules with integrable connection on $U_K$ regular along $D_K$ with exponents in $\overline{\Sigma}$ by $MIC(U_K/K)^{reg, \overline{\Sigma}}$.
\end{defi} 
\begin{prop}
The notion of exponents in $\overline{\Sigma}$ for a module with integrable connection on $U_K$ regular along $D_K$ is independent of the choice of the \'etale neighborhood and the diagram chosen in 
definition \ref{thetabar}.
\end{prop}
\begin{proof}
We can proceed as in proposition \ref{regindip}, using proposition \cite{AndBal01} I.6.4.3.
\end{proof}

If we take a set $\overline{\Sigma}=\prod_{h=1}^t \overline{\Sigma}_h \subset \bar{K}^t/\mathbb{Z}^t$ and a section $\tau:\bar{K}/\mathbb{Z} \rightarrow \bar{K}$ of the canonical projection, the log extension theorem of Andr\'e and Baldassarri (\cite{AndBal01} I. 4.9) reads as follows.
\begin{teo}\label{P}
Given $\overline{\Sigma}=\prod_{h=1}^t \overline{\Sigma}_h \subset (\bar{K}/\mathbb{Z})^t$, where $t$ is the number of irreducible components of the divisor $D_K=\cup_{h=1}^tD^{h;K}$. For every  section $\tau:\bar{K}/\mathbb{Z}\rightarrow \bar{K}$, the restriction functor
$$ P_{\tau}: MIC(X_K/K)^{\log, \tau(\overline\Sigma)}\longrightarrow MIC(U_K/K)^{reg, \overline{\Sigma}}, $$
which is induced by the tensor functor 
$$ P: MIC(X_K/K)^{\log}\longrightarrow MIC(U_K/K)^{reg}, $$
is a well-defined functor which is 
an equivalence of categories, where $\tau(\overline\Sigma) := 
\prod_{h=1}^t \tau(\overline{\Sigma}_h)$. 
\end{teo}
\begin{proof}
We can work \'etale locally, apply \cite{AndBal01} I.4.9, and use \'etale descent.
\end{proof}
\begin{prop}\label{tannakianalg}
If $\overline{\Sigma}$ is a subgroup of $\bar{K}^t/\mathbb{Z}^t$, 
the category $MIC(U_K/K)^{reg, \overline{\Sigma}}$ is a rigid abelian tensor category. 
%% closed by extensions and subquotients and its objects have internal hom.   
\end{prop}
\begin{proof}
It follows from proposition I.6.4.2 of \cite{AndBal01}.
\end{proof}

Finally in this section, we recall the results in \cite{DiP12} on 
the functors relating the categories 
$I_{\mathrm{conv}}(X_k/V)^{\log}$, 
$\widehat{MIC}(\hat{X}/V)^{\log}$ and 
$MIC(X_K/K)^{\log}$. 

\begin{teo}\label{phigamma}
Let $\overline{\Sigma} = \prod_{h=1}^t \Sigma_h$ be a 
subset of $(\mathbb{Z}_p/\mathbb{Z})^t$ and $\tau: \mathbb{Z}_p/\mathbb{Z} 
\rightarrow \mathbb{Z}_p$ be a section. Then 
there exists a fully faithful functor 
$$ \Phi_{\tau}: I_{\mathrm{conv}}(X_k/V)^{\log, \tau(\overline{\Sigma})}
\longrightarrow \widehat{MIC}(\hat{X}/V)^{\log,\tau(\overline{\Sigma})} $$
which is induced by the fully faithful tensor functor 
$$ \Phi: I_{\mathrm{conv}}(X_k/V)^{\log}
\longrightarrow \widehat{MIC}(\hat{X}/V)^{\log} $$
of `forgetting the convergence condition'. 
Also, there exists an equivalence 
$$ \Gamma_{\tau}: MIC(X_K/K)^{\log, \tau(\overline{\Sigma})}
\longrightarrow \widehat{MIC}(\hat{X}/V)^{\log,\tau(\overline{\Sigma})} $$
which is induced by the GAGA equivalence 
$$ \Gamma: MIC(X_K/K)^{\log}
\longrightarrow \widehat{MIC}(\hat{X}/V)^{\log}. $$
\end{teo}

\begin{proof}
The functor $\Phi$ is the composition of 
the fully faithful tensor functor $\tilde{\Phi}$ 
in section 5 of \cite{DiP12} and the equivalence in 
proposition 2, section 5 in \cite{DiP12}. 
The equivalence $\Gamma$ is the equivalence 
$\tilde{\Psi}$ in section 5 of \cite{DiP12}. 
Also, it is easy to see that $\Phi, \Gamma$ induce the functors 
$\Phi_{\tau}, \Gamma_{\tau}$ respectively. 
\end{proof}

\section{Properties of algebraization functor}\label{Tannakian properties of the algebraization functor}
The main result of \cite{DiP12} can be translated in the context of this paper as follows: if $\overline{\Sigma}$ is (NLG), then we constructed 
a fully faithful functor 
\begin{equation}
\psi_{\tau}:I^{\dag}(U_k/V)^{\log, \overline{\Sigma}}\longrightarrow 
MIC(U_K/K)^{reg, \overline{\Sigma}}, 
\end{equation}
which is the composition of the four functors 
\begin{align*}
& j^{\dag,-1}_{\tau}: I^{\dag}(U_k/V)^{\log, \overline{\Sigma}} 
\longrightarrow 
I_{\rm conv}(X_k/V)^{\log,\tau(\overline{\Sigma})}, \\ 
& \Phi_{\tau}: I_{\mathrm{conv}}(X_k/V)^{\log, \tau(\overline{\Sigma})}
\longrightarrow \widehat{MIC}(\hat{X}/V)^{\log,\tau(\overline{\Sigma})}, \\
& \Gamma^{-1}_{\tau}: \widehat{MIC}(\hat{X}/V)^{\log,\tau(\overline{\Sigma})} 
\longrightarrow 
MIC(X_K/K)^{\log, \tau(\overline{\Sigma})}, \\ 
& P_{\tau}: MIC(X_K/K)^{\log, \tau(\overline{\Sigma})} 
\longrightarrow MIC(U_K/K)^{reg, \overline{\Sigma}}, 
\end{align*}
where the first one is the quasi-inverse of the functor 
$j^{\dag}_{\tau}$ in theorem \ref{mainstable}, 
the second one is as in theorem \ref{phigamma}, 
the third one is the quasi-inverse of the functor 
$\Gamma_{\tau}$ in theorem \ref{phigamma} and 
the fourth one is the functor in theorem \ref{P}. 

We want to stress that there are tensor functors 
\begin{align*}
& j^{\dag}: I_{\mathrm{conv}}(X_k/V)^{\log} \longrightarrow 
I^{\dag}(U_k/V)^{\log}, \\
& \Phi: I_{\mathrm{conv}}(X_k/V)^{\log, \tau(\overline{\Sigma})}
\longrightarrow \widehat{MIC}(\hat{X}/V)^{\log,\tau(\overline{\Sigma})}, \\
& \Gamma^{-1}: \widehat{MIC}(\hat{X}/V)^{\log} 
\longrightarrow 
MIC(X_K/K)^{\log}, \\ 
& P: MIC(X_K/K)^{\log} 
\longrightarrow MIC(U_K/K)^{reg}, 
\end{align*}
extending $j_{\tau}^{\dag}, \Phi_{\tau}, \Gamma^{-1}_{\tau}, 
P_{\tau}$ respectively. 

The first main theorem in this paper is the following. 

\begin{teo}\label{teotensor}
The functor $\psi_{\tau}$ is a tensor 
functor which is independent of the choice of $\tau$. 
\end{teo}

To prove the theorem, first we prove a certain property 
close to the full-faithfulness of the functor $j^{\dag}$. 

\begin{prop}\label{taupertensore}
Let $\Sigma_i := \prod_{h=1}^t \Sigma_{i,h} \, (i=1,2)$ 
be subsets of $\mathbb{Z}_p^t$ which are 
(NLD) and (NID) and suppose that, 
for every $h$ and every $\xi_i \in \Sigma_{i,h} \,(i=1,2)$, 
$\xi_2 - \xi_1$ belongs to $\mathbb{Z}_p \setminus \mathbb{Z}_{>0}$. 
Let $\mathcal{E}_i$ be an object in 
$I_{\rm conv}(X_k/V)^{\log, \Sigma_i}$ for $i=1,2$. Then 
the functor $j^{\dag}$ induces the isomorphism 
$$\mathrm{Hom}(\mathcal{E}_1, \mathcal{E}_2) \cong 
\mathrm{Hom}(j^{\dagger}\mathcal{E}_1, j^{\dagger}\mathcal{E}_2).$$
\end{prop}

\begin{proof} 
The argument is analogous to the one in proposition 1.18 of \cite{Shiparalog}.

Given a morphism $f: j^{\dagger}\mathcal{E}_1 \rightarrow 
j^{\dagger}\mathcal{E}_2$ of log overconvergent isocrystals, 
we should prove that $f$ extends uniquely to a morphism 
$\tilde{f}: \mathcal{E}_1 \rightarrow \mathcal{E}_2$. 
We can work \'etale locally with the diagram \eqref{etale locally for}. 
Let us take a small enough strict neighborhood $W$ of $]U_k[_{\hat{X}}$ in $]X_k[_{\hat{X}}$ so that $f$ induces a morphism 
$\phi: E'_1 \rightarrow E'_2$, where  $E'_i \,(i=1,2)$ 
is the module with integrable connection induced by $j^{\dagger}\mathcal{E}_i$ on $W$. We denote by $E_i \,(i=1,2)$ 
the log-$\nabla$-module induced by $\mathcal{E}_i$ 
on $\hat{X}_K$ (with respect to $y_1, ..., y_s$). 
It suffices to extend the morphism $\phi$ uniquely to a morphism 
$E_1 \rightarrow E_2$. 

Let us take the covering
$$\hat{X}_K=\bigcup_{J\subset \{1\dots s\}}A_J, $$
of $\hat{X}_K$, where 
\begin{align*}
A_J& =\{P\in \hat{X}_K|\,\,\, |y_j(P)|<1 \,\,\,(j\in J)\,\,\, |y_j(P)|\geq \lambda\,\,\, (j \notin J) \} \\ 
& = \{P\in \hat{X}_K|\,\,\, y_j(P)=0 \,\,\,(j\in J),\,\,\,|y_j(P)|\geq \lambda\,\,\, (j \notin J)\}\times A^{|J|}[0,1). 
\end{align*}
Also, let us take $\lambda \in (0,1)\cap \Gamma^*$ so that $E'_i$'s are 
defined on the following set:
$$B=\{P\in \hat{X}_K|\,\,\, |y_j(P)|\geq \lambda \,\,\, \forall j \}.$$
The covering of $\hat{X}_{K}$ given by the $A_J$'s restricts to 
the covering $B=\bigcup_{J\subset\{1,\dots,s\}}B_J$, where 
\begin{align*}
B_J& =\{P\in \hat{X}_K|\,\,\, \lambda \leq |y_j(P)|<1 \,\,\,(j\in J),\,\,\,|y_j(P)|\geq \lambda\,\,\, (j \notin J)\} \\ 
& = \{P\in \hat{X}_K|\,\,\, y_j(P)=0 \,\,\,(j\in J)\,\,\, |y_j(P)|\geq \lambda\,\,\, (j \notin J)\}\times A^{|J|}[\lambda,1). 
\end{align*}
We can check that $E_1, E_2$ are log-convergent 
in the sense of \cite{Shilogext} on 
\begin{equation}\label{Jfinoa1}
A_J = \{P\in \hat{X}_K| y_j(P)=0 \,\,\,(j\in J),\,\,\,|y_j(P)|\geq \lambda\,\,\, (j \notin J)\}\times A^{|J|}[0,1)
\end{equation}
in the same way as proposition 3.6 of \cite{Shilogext} and they have exponents 
in $\prod_j \Sigma_{1,h(j)}$, $\prod_j \Sigma_{2,h(j)}$, respectively. So, 
by proposition 2.12 of \cite{Shilogext}, they are 
$\prod_j \Sigma_{1,h(j)}$-unipotent, $\prod_j \Sigma_{2,h(j)}$-unipotent 
respectively. Hence, by proposition \ref{sigma1sigma2} below, 
the canonical map 
\begin{equation}\label{restrhomtensor1}
\mathrm{Hom}_{A_J}(E_1, E_2) 
\rightarrow \mathrm{Hom}_{B_J}(E'_1, E'_2) 
\end{equation}
is an isomorphism. So the restriction $\phi_J$ of $\phi$ to $B_J$ 
extends uniquely to a morphism $E_1 \rightarrow E_2$ on $A_J$ for every $J$. 
On $A_I\cap A_J$ we have the extensions $\phi_I$ and $\phi_J$, which glue 
again by proposition \ref{sigma1sigma2} below 
because they coincide on the set 
\begin{equation*}
\begin{split}
B_I\cap B_J=\{P\in \hat{X}_K|\,\,\,  \lambda  \leq |y_j(P)| < 1 & \,\,\,(j\in (I\cup J)-(I\cap J)),\\
& |y_j(P)|\geq \lambda \,\,\, (j \notin  (I\cup J))\}\times A^{|I\cap J|}[\lambda,1), \\
\end{split}
\end{equation*}
both extending the restriction of $\phi$ on $B_I\cap B_J.$
So we obtain the desired extension of $\phi$ to $\hat{X}_K$. 
\end{proof}

\begin{prop}\label{sigma1sigma2}
Let $Y$ be a smooth rigid space over $K$, let $\lambda$ be an element in $(0,1)\cap \Gamma^*$ and let $\Sigma_1:=\prod_{j=1}^n\Sigma_{1,j}$, $\Sigma_2 :=\prod_{j=1}^n\Sigma_{2,j}$ be subsets of $\mathbb{Z}_p^n$. For $i=1,2$, let $E_i$ be a $\Sigma_i$-unipotent log-$\nabla$-module on $Y\times A^n_K[0,1)$ 
%%with respect to $t_1, \dots t_n$ (where $t_1, \dots t_n$ are the coordinates of $A^n_K[0,1)$) 
and let us put $E_i'=E_{i|Y\times A^n_K[\lambda, 1)}, $ 
which is a log-$\nabla$-module on $Y\times A^n_K[\lambda, 1).$ If for any $j$ and for any $\xi_{1} \in \Sigma_{1,j, E_1}$,  $\xi_{2} \in \Sigma_{2,j, E_2}$ we have $\xi_{2}-\xi_{1} \in \mathbb{Z}_p\setminus \mathbb{Z}_{>0},$ 
the restriction functor induces an isomorphism of 
the following groups of homomorphisms of log-$\nabla$-modules
$$\mathrm{Hom}(E_1, E_2)\cong \mathrm{Hom}(E_1', E_2').$$
\end{prop}
\begin{proof}
This is proposition 1.6 of \cite{Shiparalog}.
\end{proof}

We prove the independence of the functor $\psi_{\tau}$ of the choice 
of $\tau$. 

\begin{prop}\label{teotensor1}
The functor $\psi_{\tau}$ is independent of the choice of $\tau$. 
\end{prop}

\begin{proof}
Let us take sections $\tau_1, \tau_2: 
\mathbb{Z}_p/\mathbb{Z}\rightarrow \mathbb{Z}_p$. 
It suffices to prove the existence of the canonical 
isomorphism of functors $\psi_{\tau_1} \cong \psi_{\tau_2}$. 
Define $\tau_3: \mathbb{Z}_p/\mathbb{Z}\rightarrow \mathbb{Z}_p$ by 
$$ 
\tau_3(\xi) := 
\begin{cases} 
\tau_1(\xi), & \text{if $\tau_1(\xi) - \tau_2(\xi) 
\in \mathbb{Z}_{\leq 0}$}, \\ 
\tau_2(\xi), & \text{if $\tau_1(\xi) - \tau_2(\xi) \in \mathbb{Z}_{> 0}$}. 
\end{cases}
$$ 
Then, for any $\xi \in \mathbb{Z}_p/\mathbb{Z}$, 
$\tau_3(\xi) - \tau_1(\xi), \tau_3(\xi) - \tau_2(\xi)$ belong to 
$\mathbb{Z}_p \setminus \mathbb{Z}_{>0}$. Hence, by replacing 
$\tau_1$ or $\tau_2$ by $\tau_3$, we are reduced to proving the 
isomorphism $\psi_{\tau_1} \cong \psi_{\tau_2}$ in the case where 
$\tau_2(\xi) - \tau_1(\xi) \in \mathbb{Z}_p \setminus \mathbb{Z}_{>0}$ 
for any $\xi \in \mathbb{Z}_p/\mathbb{Z}$. 

Let us take an object $\mathcal{E}$ in 
$I^{\dag}(U_k/V)^{\log, \overline{\Sigma}}$ and put 
$\mathcal{E}_i := j_{\tau_i}^{\dagger, -1}(\mathcal{E})$. 
Then, by proposition \ref{taupertensore}, 
we have the unique morphism 
$\mu: \mathcal{E}_1 \rightarrow \mathcal{E}_2$ in 
$I_{\mathrm{conv}}(X_k/V)^{\log}$ which extends the identity on 
$\mathcal{E}$. By applying $P \circ \Gamma^{-1} \circ \Phi$, 
we obtain the morphism 
$f: \psi_{\tau_1}(\mathcal{E}) \rightarrow \psi_{\tau_2}(\mathcal{E})$. 
It suffices to prove that $f$ is an isomorphism. 

For an integer $\alpha$, we can define 
the log convergent isocrystal $\mathcal{O}(\alpha D_k)$ 
as in \cite{Shiparalog} proposition 3.1, in the following way: 
We denote by $\tau^{\alpha}:\mathbb{Z}_p/\mathbb{Z}\rightarrow \mathbb{Z}_p$ 
any section with $\tau^{\alpha}(0) = -\alpha$. 
By theorem \ref{mainstable}, 
if $\bold{0}=(\overline{0}, \dots, \overline{0})$, we know that the functor 
$$j_{\tau^{\alpha}}^{\dag}:I_{\mathrm{conv}}(\hat{X}/V)^{\log,\tau^{\alpha}(\bold{0})}\longrightarrow I^{\dag}(U_k/V)^{\log, \bold{0}}$$
is an equivalence of categories. 
Hence we have the unique object in 
$I_{\mathrm{conv}}(\hat{X}/V)^{\log, \tau^{\alpha}(\boldsymbol{0})}$ 
which extends the trivial log overconvergent isocrystal 
in $I^{\dag}(U_k/V)^{\log, \boldsymbol{0}}$. We denote it 
by $\mathcal{O}(\alpha D_k)$. By considering the description of it 
in terms of (formal) log connections as in \cite{Shiparalog} proposition 
3.1, we see that 
$\Gamma^{-1} \circ \Phi$ sends $\mathcal{O}(\alpha D_k)$ to the object 
$(\mathcal{O}_{X_K}(\alpha D_K), d)$ in 
$MIC(X_K/K)^{\log}$, and so it is sent to the trivial object in 
$MIC(U_K/K)^{reg}$ by the functor $P$. 
Also, when $\alpha \geq 0$, the unique morphism 
$\iota: \mathcal{O} \rightarrow \mathcal{O}(\alpha D_k)$ in 
$I_{\mathrm{conv}}(\hat{X}/V)^{\log}$ (where $\mathcal{O}$ is the trivial 
object) extending the identity (whose existence follows from 
proposition \ref{taupertensore}) is sent to the identity map. 

Now, let $\overline{\Sigma}_{\mathcal{E}} = 
\prod_h \overline{\Sigma}_{\mathcal{E},h}$ as in definition \ref{expoperE} and 
let $\alpha$ be an element in $\Z_{\geq 0}$ satisfying the 
following condition: for any $h$ and any $\xi_i \in 
\tau_i(\overline{\Sigma}_{\mathcal{E},h}) \, (i=1,2)$, 
$\xi_1 - \xi_2 - \alpha$ belongs to $\mathbb{Z}_p \setminus \mathbb{Z}_{>0}$. 
(Such $\alpha$ exists because $\overline{\Sigma}_{\mathcal{E}}$ is a 
finite set.) If we put $\mathcal{E}_1(\alpha D_k) := 
\mathcal{E}_1 \otimes \mathcal{O}(\alpha D_k)$, it is an object 
in $I_{\rm conv}(\hat{X}/V)^{\log, \tau_1(\overline{\Sigma}_{\mathcal{E}}) - \boldsymbol{\alpha}}$  
(where $\boldsymbol{\alpha}=(\alpha, \dots, \alpha)$) and 
$\mathcal{E}_2$ is an object 
in $I_{\rm conv}(\hat{X}/V)^{\log, \tau_2(\overline{\Sigma}_{\mathcal{E}})}$. 
Hence, by proposition \ref{taupertensore}, we have the unique 
morphism $\mu': \mathcal{E}_2 \rightarrow \mathcal{E}_1(\alpha D_k)$ 
extending the identity on $\mathcal{E}$, and by sending it by 
$P \circ \Gamma^{-1} \circ \Phi$, we obtain the morphism 
$g: \psi_{\tau_2}(\mathcal{E}) \rightarrow 
\psi_{\tau_1}(\mathcal{E}) \otimes (P \circ \Gamma^{-1} \circ \Psi)(\mathcal{O}(\alpha D_k)) = \psi_{\tau_1}(\mathcal{E})$. 
By the unicity of 
$\mu$ and $\mu'$, the composite $\mu' \circ \mu$ is equal to 
the map ${\rm id}_{\mathcal{E}_1} \otimes \iota$. Hence, by 
applying $P \circ \Gamma^{-1} \circ \Phi$, we see that the composite $g \circ f$ is equal to the identity map. 

Now we repeat the same argument: we have the unique morphism 
$\mu'': \mathcal{E}_2(-\alpha D_k) \rightarrow \mathcal{E}_1$ 
extending the identity on $\mathcal{E}$, which is by unicity 
equal to $\mu' \otimes {\rm id}_{\mathcal{O}(-\alpha D_k)}$. 
By applying $P \circ \Gamma^{-1} \circ \Phi$, we obtain the morphism 
$g: \psi_{\tau_2}(\mathcal{E}) \rightarrow \psi_{\tau_1}(\mathcal{E})$ 
same as before. By the argument in the previous paragraph, we see 
that the composite $f \circ g$ is also equal to the identity map. 
Hence we see that $f$ is an isomorphism, and so we are done. 
\end{proof}

Next we prove that the functor $\psi_{\tau}$ is a tensor functor. 

\begin{prop}\label{teotensor2}
The functor $\psi_{\tau}$ is a tensor functor. 
\end{prop}

\begin{proof}
Let $\mathcal{E}_i \,(i=1,2)$ be objects in 
$I^{\dag}(U_k/V)^{\log, \overline{\Sigma}}$ and 
let $\overline{\Sigma}_{\mathcal{E}_i} := 
\prod_h \overline{\Sigma}_{\mathcal{E}_i,h} \,(i=1,2), 
\overline{\Sigma}_{\mathcal{E}_1 \otimes \mathcal{E}_2} := 
\prod_h \overline{\Sigma}_{\mathcal{E}_1 \otimes \mathcal{E}_2,h}$ 
be as in definition \ref{expoperE}. 
Let us take a section $\tau': \mathbb{Z}_p/\mathbb{Z} \rightarrow 
\mathbb{Z}_p$ satisfying the following condition: For any $h$ and 
for any $a_{i} \in \overline{\Sigma}_{\mathcal{E}_i,h} \,(i=1,2), b \in 
\overline{\Sigma}_{\mathcal{E}_1 \otimes \mathcal{E}_2,h}$, 
$\tau(a_1) + \tau(a_2) - \tau'(b) \in \mathbb{Z}_p \setminus 
\mathbb{Z}_{>0}$. (It is possible to take such $\tau$ 
because $\overline{\Sigma}_{\mathcal{E}_i}, 
\overline{\Sigma}_{\mathcal{E}_1 \otimes \mathcal{E}_2}$ 
are finite sets.) Put 
$\mathcal{F}_1 := j_{\tau'}^{\dagger,-1}(\mathcal{E}_1 \otimes 
\mathcal{E}_2)$, 
$\mathcal{F}_2 := j_{\tau}^{\dagger,-1}(\mathcal{E}_1) \otimes 
 j_{\tau}^{\dagger,-1}(\mathcal{E}_2)$. 
Then, $\mathcal{F}_1$ belongs to 
$I_{\rm conv}(X_k/V)^{\log, \tau'(\overline{\Sigma}_{\mathcal{E}_1 \otimes \mathcal{E}_2})}$ and $\mathcal{F}_2$ belongs to 
$I_{\rm conv}(X_k/V)^{\log, \tau(\overline{\Sigma}_{\mathcal{E}_1})+ \tau(\overline{\Sigma}_{\mathcal{E}_2})}$. Hence, by definition of $\tau'$, 
there exists the unique morphism 
$\mu: \mathcal{F}_1 \rightarrow \mathcal{F}_2$ extending the 
identity morphism on $\mathcal{E}_1 \otimes \mathcal{E}_2$. 
By applying $P \circ \Gamma^{-1} \circ \Phi$, we obtain the morphism 
$f: \psi_{\tau'}(\mathcal{E}_1 \otimes \mathcal{E}_2) \rightarrow 
\psi_{\tau}(\mathcal{E}_1) \otimes \psi_{\tau}(\mathcal{E}_2)$. 

Then, by the same argument as the proof of the previous 
proposition (using $\mathcal{O}(\alpha D_k)$ for some $\alpha$), we see that 
the morphism $f$ is an isomorphism. By combining this with 
the isomorphism 
$\psi_{\tau'}(\mathcal{E}_1 \otimes \mathcal{E}_2) \cong 
\psi_{\tau}(\mathcal{E}_1 \otimes \mathcal{E}_2)$ of the previous 
proposition, we obtain the isomorphism 
$$ \psi_{\tau}(\mathcal{E}_1 \otimes \mathcal{E}_2) \cong
\psi_{\tau}(\mathcal{E}_1) \otimes \psi_{\tau}(\mathcal{E}_2), $$
as desired. 
\end{proof}

By propositions \ref{teotensor1} and \ref{teotensor2}, 
the proof of theorem \ref{teotensor} is finished. 

Next we prove the following theorem, which is the second main 
result in this paper. 

\begin{teo}\label{extsq}
The essential image of the functor 
$\psi_{\tau}$ is closed under 
extensions and subquotients. 
\end{teo}

In order to prove this theorem, it suffices to prove the same property 
for the functor 
$$\Phi_{\tau}: I_{\rm conv}(X_k/V)^{\log, \tau(\overline{\Sigma})}\longrightarrow \widehat{MIC}(\hat{X}/V)^{\log, \tau(\overline{\Sigma})}, $$
because the functors $j^{\dag,-1}_{\tau}, \Gamma^{-1}_{\tau}, P_{\tau}$ are 
equivalences. 

Also, it is easy to see that the condition on exponent is closed 
under extensions and subquotients. Hence it suffices to prove the 
same property for the functor 
$$\Phi: I_{\mathrm{conv}}(X_k/V)^{\log}\longrightarrow 
\widehat{MIC}(\hat{X}/V)^{\log}. $$
Then, by proposition 8 of \cite{DiP12}, one can replace 
$\hat{X}$ by $\hat{U}$ to prove it. 

Now we recall the definition of formal 
stratifications and special formal stratifications 
(definition 3.2.10 of \cite{Shi00}, definition 15 of \cite{DiP12}). 

\begin{defi}
We define the category $\widehat{Str}(\hat{U}/V)$ of 
formal stratifications on $\hat{U}$ as follows: 
An object is a pair $(E,\{\epsilon_n\})$ of a locally free isocoherent 
sheaf $E$ on $\hat{U}$ and 
a compatible 
family $\{\epsilon_n: \mathcal{O}_{\hat{U}^n} \otimes E \rightarrow 
E \otimes \mathcal{O}_{\hat{U}^n}\}_{n \in \mathbb{N}}$ 
of linear isomorphisms 
(where $\hat{U}^n$ is the $n$-th infinitesimal neighborhood of 
$\hat{U}$ in $\hat{U} \times_V \hat{U}$) satisfying the cocycle 
condition on the infinitesimal neighborhoods of 
$\hat{U}$ in $\hat{U} \times_V \hat{U} \times_V \hat{U}$. 
\end{defi}
\begin{defi}\label{special}
Let $(E,\{\epsilon_n\})$ be an object of $\widehat{Str}(\hat{U}/V)$ and let $\tilde {E}$ be a coherent $p$-torsion-free $\mathcal{O}_{\hat{U}}$-module such that $K\otimes \tilde{E}=E$; we say that $(E,\{\epsilon_n\})$ is special if there exists a sequence of integers $k(n)$ for $n \in \mathbb{N}$ such that:
\begin{itemize}
\item[(i)]$k(n)=o(n)$ for $n\rightarrow \infty$,
\item[(ii)]the restriction of the map $p^{k(n)}\epsilon_n$ to 
$\mathcal{O}_{\hat{U}^n} \otimes \tilde{E}$ has 
image contained in $\tilde{E} \otimes \mathcal{O}_{\hat{U}^n}$ and the restriction of the map $p^{k(n)}\epsilon_n^{-1}$ to $\tilde{E} \otimes \mathcal{O}_{\hat{U}^n}$ has image contained in $\mathcal{O}_{\hat{U}^n} \otimes \tilde{E}$.
\end{itemize}
\end{defi}

It is known that the definition of being special is 
independent of the choice of a coherent $p$-torsion-free $\mathcal{O}_{\hat{U}}$-module $\tilde{E}$. Hence the property of being special is a local 
property. 

We proved in theorem 3.2.15 of \cite{Shi00} 
and in section 6 of \cite{DiP12} the following theorem, 
which characterizes 
the essential image of the functor $\Phi$ on $\hat{U}$. 

\begin{teo}\label{sp}
We have an equivalence of categories 
\begin{equation}\label{micstr}
\widehat{MIC}(\hat{U}/V) \cong \widehat{Str}(\hat{U}/V), 
\end{equation}
and an object in $\widehat{Str}(\hat{U}/V)$ belongs to 
the essential image of 
$\Phi: I_{\rm conv}(U_k/V) \longrightarrow 
\widehat{MIC}(\hat{U}/V) \cong \widehat{Str}(\hat{U}/V)$ if and only if 
it is special. 

When $x_1, \dots, x_l$ are local coordinates of $\hat{U}$ and 
$\partial_1, \dots, \partial_l$ denote the corresponding derivations, 
the image of $(E,\nabla)$ by \eqref{micstr} 
is given by $(E,\{\epsilon_n\})$ with 
$$ \epsilon_n = \sum_{|{\boldsymbol{\beta}}| \leq n} 
\frac{1}{{\boldsymbol{\beta}}!} \nabla_{{\boldsymbol{\beta}}} \cdot 
(1 \otimes x_i - x_i \otimes 1)^{{\boldsymbol{\beta}}}, \quad 
\nabla_{\boldsymbol{\beta}} := 
(\textrm{id}\otimes \partial_{1}\circ \nabla )^{\beta_1}\circ \dots \circ (\textrm{id}\otimes \partial_{l}\circ \nabla )^{\beta_l}.$$
\end{teo}

Note that $\epsilon_n^{-1}$ is given by 
$$ \epsilon_n^{-1} = \sum_{|{\boldsymbol{\beta}}| \leq n} 
\frac{(-1)^{|{\boldsymbol{\beta}}|}}{{\boldsymbol{\beta}}!} \nabla_{{\boldsymbol{\beta}}} \cdot 
(1 \otimes x_i - x_i \otimes 1)^{{\boldsymbol{\beta}}} $$
in the above situation. 

By the above theorem, 
the proof of theorem \ref{extsq} is reduced to the claim that the 
category of special formal stratifications is closed under 
taking extensions and subquotients, which we prove in the following 
two lemmata. 

\begin{lemma}\label{estensioni}
Let $(E,\{e_n\}), (F,\{f_n\}), (G,\{g_n\})$ be objects in 
${\widehat{Str}}(\hat{U}/V)$ with 
$(E,\{e_n\})$ and $(F,\{f_n\})$ special and suppose that 
there exists an exact sequence 
\begin{equation}\label{exstr}
0\longrightarrow(E,\{e_n\})\longrightarrow(G,\{g_n\})\longrightarrow
(F,\{f_n\})\longrightarrow0
\end{equation}
in $\widehat{Str}(\hat{U}/V)$. 
Then $(G,\{g_n\})$ is also special.
\end{lemma}
\begin{proof}
We can restrict to the local situation 
because being special is a local property. 
So we can take local coordinates $x_1, \dots, x_l$ of $\hat{U}$ and 
the corresponding derivations $\partial_1, \dots, \partial_l$. 
Also, we can suppose that $F$ is a projective 
$K\otimes\mathcal{O}_{\hat{X}}$-module. Then 
the exact sequence \eqref{exstr} splits as a sequence of 
isocoherent sheaves. We fix a splitting. Then we have an equality 
$G = E \oplus F$ as $K\otimes\mathcal{O}_{\hat{X}}$-modules. 
Also, we take $p$-torsion free coherent 
$\mathcal{O}_{\hat{X}}$-modules $\tilde{E}$ and $\tilde{F}$ with $K\otimes \tilde{E}=E$ and $K\otimes \tilde{F}=F$. Then 
we have the isomorphisms 
$K\otimes(\tilde{E}\oplus\tilde{F})= E \oplus F = G$. 

Let us denote by $(E,\nabla_E)$, $(F,\nabla_F)$ and $(G,\nabla_G)$ 
the modules with integrable connections associated to 
$(E,\{e_n\})$, $(F,\{f_n\})$ and $(G,\{g_n\})$ respectively. 
Then, by the exact sequence \eqref{exstr} and the splitting in the previous paragraph, we can write $\nabla_G(\partial_{i})$ as the matrix 
$$
 \nabla_G(\partial_{i})=\left(\begin{array}{cc}\nabla_{E}(\partial_{i})&B_i\\
0&\nabla_{F}(\partial_{i})\\
\end{array}\right) 
$$
for some $B_i$. If we denote by $
\nabla_{E, \boldsymbol{\beta}}, 
\nabla_{F, \boldsymbol{\beta}}, 
\nabla_{G, \boldsymbol{\beta}}$ the operator 
$\nabla_{\boldsymbol{\beta}}$ in theorem \ref{micstr} for 
$(E,\nabla_E)$, $(F,\nabla_F), (G,\nabla_G)$ respectively, 
we can calculate $\nabla_{G, \boldsymbol{\beta}}$ as 
\begin{align}
\frac{1}{\boldsymbol{\beta}!}\nabla_{G,\boldsymbol{\beta}}
& = 
\frac{1}{\boldsymbol{\beta}!}\left(\begin{array}{cc}\nabla_{E,\boldsymbol{\beta}}&0\\
0&\nabla_{F,\boldsymbol{\beta}} \end{array}\right) \label{potenza} \\ 
& \hspace{1cm} + \sum_{i=1}^{l}\sum_{b+c=\boldsymbol{\beta}-e_i}\frac{b!c!}{\boldsymbol{\beta}!}\frac{1}{b!}\left(\begin{array}{cc}\nabla_{E,b}&0\\
0&\nabla_{F,b}
\end{array}\right)\left(\begin{array}{cc}0&B_i\\
0&0
\end{array}\right)\frac{1}{c!}\left(\begin{array}{cc}\nabla_{E,c}&0\\
0&\nabla_{F,c}
\end{array}\right), \nonumber 
\end{align}
with $e_i=(0,\dots,1,\dots,0)$ with $1$ at the $i$-th place.
We know that there exists a sequence $l(n)$ of integers with $l(n) = o(n)$ as 
$n \rightarrow \infty$ such that 
$$p^{l(n)}\frac{1}{\boldsymbol{\beta}!}\nabla_{E,\boldsymbol{\beta}}(\tilde{E}) \subset \tilde{E}, \quad 
p^{l(n)}\frac{1}{\boldsymbol{\beta}!}\nabla_{F,\boldsymbol{\beta}}(\tilde{F}) 
\subset \tilde{F}$$
for every $\boldsymbol{\beta}$ such that $|\boldsymbol{\beta}|\leq n$, because both $(E, \{e_n\})$ and $(F, \{f_n\})$ are special.
If $\alpha$ is a positive integer such that $p^{\alpha}B_i(\tilde{E}) \subset \tilde{F}$ for $e\in \tilde{E}$ for any $i=1,\dots, l$, then 
$$k(n)=l(n) + v_p(\frac{\boldsymbol{\beta}!}{b!c!}) + \alpha$$ 
is a sequence of integers such that 
$$p^{k(n)}\frac{1}{\boldsymbol{\beta}!}\nabla_{G,\boldsymbol{\beta}}
(\tilde{E}\oplus\tilde{F}) 
\subset \tilde{E}\oplus\tilde{F}$$
for every $\boldsymbol{\beta}$ such that $|\boldsymbol{\beta}|\leq n$.
We know that $l(n)=o(n)$ for $n\rightarrow \infty$ and $\alpha$ does not depend on $n$, so that the only thing that we have to show is that 
$$\max_{|\boldsymbol{\beta}|\leq n, b+c=\boldsymbol{\beta}-e_i} v_p(\frac{\boldsymbol{\beta}!}{b!c!})=o(n), $$ 
which is standard (see \cite{Shi00} proposition 5.2.14 for example). 
\end{proof}
In the next lemma we show that the category of special objects is closed by taking subquotients.
\begin{lemma}\label{subquotients}
Let us suppose that we have an exact sequence in 
$\widehat{Str}(\hat{U}/V)$:
$$0\longrightarrow(E,\{e_n\})\longrightarrow(G,\{g_n\})\longrightarrow
(F,\{f_n\})\longrightarrow0.$$
Then, if $(G,\{g_n\})$ is special, then $(E,\{e_n\})$ and 
$(F,\{f_n\})$ are special. 
\end{lemma}
\begin{proof}
As in lemma \ref{estensioni}, we have the isomorphisms 
$K\otimes(\tilde{E}\oplus\tilde{F})= E \oplus F = G$ as 
$K \otimes_V \mathcal{O}_{\hat{U}}$-modules. 
By hypothesis, there exists a sequence of integers $k(n)$ with 
$k(n)=o(n)$ as $n\rightarrow \infty$ such that 
\begin{equation}\label{g}
p^{k(n)}\frac{1}{\boldsymbol{\beta}!}\nabla_{G,\boldsymbol{\beta}}
(\tilde{E}\oplus\tilde{F}) 
\subset
 \tilde{E}\oplus\tilde{F}
\end{equation} 
for $|\boldsymbol{\beta}|\leq n$. 
By restricting \eqref{g} to $E$, we 
obtain the inclusion 
$$p^{k(n)}\frac{1}{\boldsymbol{\beta}!}\nabla_{E,\boldsymbol{\beta}}
(\tilde{E}) 
\subset 
\tilde{E}$$   
and this implies that $(E,\{e_n\})$ is special. 
Also, by sending \eqref{g} to $F$, we obtain the inclusion 
$$p^{k(n)}\frac{1}{\boldsymbol{\beta}!}\nabla_{F,\boldsymbol{\beta}}
(\tilde{F}) 
\subset 
\tilde{F}$$   
and this implies that $(F,\{f_n\})$ is also special. 
\end{proof}

Since the proof of lemmata \ref{estensioni}, \ref{subquotients}
 are finished, the proof of theorem \ref{extsq} is also finished. 

We have the following corollary as an immediate consequence of 
theorems \ref{teotensor}, \ref{extsq}. 

\begin{cor}\label{cor}
Let $\overline{\Sigma}=\prod_{h=1}^t\overline{\Sigma}_h \subset (\mathbb{Z}_p/\mathbb{Z})^t$ be a subset which satisfies (NLG) and 
let us fix a section $\tau: \mathbb{Z}_p/\mathbb{Z}\rightarrow \mathbb{Z}_p$. 
Let $\mathcal{E}$ be an object in $I^{\dag}(U_k/V)^{\log, \overline{\Sigma}}$. 
Then the functor $\psi_{\tau}$ induces an equivalence of rigid abelian tensor 
categories
\begin{equation}\label{final}
I^{\dag}(U_k/V)^{\log, \overline{\Sigma}}(\mathcal{E})\longrightarrow MIC(U_K/K)^{reg, \overline{\Sigma}}(\psi_{\tau}(\mathcal{E})), 
\end{equation}
where the left hand side (resp. the right hand side) is the 
smallest full subcategory containing $\mathcal{E}$ (resp. 
$\psi_{\tau}(\mathcal{E})$) and closed under extension, 
subquotient, tensor and dual. 
\end{cor}

Under a certain hypothesis, we have the following result which 
is a Tannakian interpretation of theorems \ref{teotensor}, \ref{extsq} and 
corollary \ref{cor}. 

\begin{teo}
Let $\overline{\Sigma}=\prod_{h=1}^t\overline{\Sigma}_h \subset (\mathbb{Z}_p/\mathbb{Z})^t$ be a subset which satisfies (NLG) and 
let us fix a section $\tau: \mathbb{Z}_p/\mathbb{Z}\rightarrow \mathbb{Z}_p$. 
If there exists a $K$-rational point $x$ of $U_K$ and $U_K$ is connected, then 
the functor 
$$\psi_{\tau}:I^{\dag}(U_k/V)^{\log, \overline{\Sigma}}\longrightarrow MIC(U_K/K)^{reg, \overline{\Sigma}}$$
is a functor of neutral 
Tannakian categories inducing the surjection of Tannaka 
duals. Also, for an object $\mathcal{E}$ in 
$I^{\dag}(U_k/V)^{\log, \overline{\Sigma}}$, the functor 
$$\psi_{\tau}:I^{\dag}(U_k/V)^{\log, \overline{\Sigma}}(\mathcal{E})\longrightarrow MIC(U_K)^{reg, \overline{\Sigma}}(\psi_{\tau}(\mathcal{E}))$$
is an equivalence of neutral Tannakian categories (thus induces the isomorphism 
of Tannaka duals). 
\end{teo}
\begin{proof}
Thanks to theorems \ref{teotensor}, \ref{extsq}, 
corollary \ref{cor} and proposition 2.21 in \cite{DelMil82}, 
it is enough to show the existence of a fiber functor from 
$I^{\dag}(U_k/V)^{\log, \overline{\Sigma}}, 
MIC(U_K/K)^{reg, \overline{\Sigma}}$ to the category of 
finite dimensional $K$-vector spaces. 
Let $x$ be the $K$-rational point that exists by hypothesis. 
We define a functor $\omega: MIC(U_K/K)^{reg, \overline{\Sigma}}\rightarrow Vec_K$ by $(\mathcal{F},\nabla)\mapsto\mathcal{F}|_x$, which is an exact tensor functor because $\mathcal{F}$ is locally free. To prove that is fully faithful we can use corollary 2.10 of \cite{Del90}, because $End(\mathcal{O}_{U_K}, d)=K$ 
(which is proven for example in 
claim 1 of proposition 3.16 of \cite{Shi00}). 
Also, if we define the functor 
$\omega':I^{\dag}(U_k/V)^{\log, \overline{\Sigma}} 
\rightarrow Vec_K$ by $\omega' := \omega \circ \psi_{\tau}$, it is 
also a fiber functor because $End(\mathcal{O})=K$ also in 
the category $I^{\dag}(U_k/V)^{\log, \overline{\Sigma}}$ because of 
the full faithfulness on $\psi_{\tau}$. 
\end{proof}

\section*{Acknowledgements}
The main part of this work was done when the first author was at the Graduate School of Mathematical Sciences of the University of Tokyo supported by a postdoctoral fellowship and kaken-hi (grant-in-aid) of the Japanese Society for the Promotion of Science (JSPS). She is now supported by a postdoctoral fellowship of Labex IRMIA. 
When the main part of this work was done, 
the second author was supported by JSPS Grant-in-Aid for 
Young Scientists (B) 21740003. 
Currently he is supported by JSPS 
Grant-in-Aid for Scientific Research (C) 25400008, 
Grant-in-Aid for Scientific Research (B) 22340001 and 
23340001.

\bibliographystyle{amsalpha}

\end{document}